\documentclass{amsart}
\usepackage[latin1]{inputenc}

\usepackage{amssymb}
\usepackage{amsmath,enumerate,rotate}
\usepackage{amssymb,latexsym}
\usepackage{epsfig}
\usepackage{graphicx}
\usepackage[english]{babel}
\usepackage{color}

\usepackage{a4wide}

\newtheorem{definition}{Definition}[section]
\newtheorem{theorem}[definition]{Theorem}
\newtheorem{lemma}[definition]{Lemma}
\newtheorem{proposition}[definition]{Proposition}
\newtheorem{corollary}{Corollary}
\newtheorem{remark}{Remark}

\def \QQ{ {\mathcal{S}}} 

\newcommand{\qq}{{\mu}} 

\def \frH {{H}}
\def \fHH {\mathfrak{H}}

\def \CR {\mathcal{R}}

\def\CF { \mathcal{F}}
\def\CL { \mathcal{L}}

\def\CB {\mathcal{B}}

\def\CL { \mathcal{L}}

\def\CS {\mathcal{S}}
\def\CG {\mathcal{G}}

\def \CS {\mathcal{S}}
\def\CE {E}

\def\a {\alpha}
\def\b {\beta}
\def\g {\gamma}

\def\eps {\epsilon}
\def\d {\delta}
\def\z {\zeta}

\def\lm {\lambda}

\def\s {\sigma}

\def\J {\mathbb{I}}

\def\E {\mathbb{E}}
\def\RE {\mathbb{R}}

\usepackage{verbatim}

\begin{document}

\title{Large Deviations for the solution of a Kac-type kinetic equation}

\author{Federico Bassetti, Lucia Ladelli}


\begin{abstract}
The aim of this paper is to study large deviations for the self-similar solution of a Kac-type kinetic equation. Under the assumption 
that the initial condition belongs to the domain of normal attraction of a stable law of index $\alpha <2$ and under suitable assumptions 
on the collisional kernel, precise asymptotic behavior of the large deviations probability is given. 
\end{abstract}

\maketitle

\section{Introduction}

This paper deals with the probability of large deviations for the  solutions 
of a class of one dimensional Boltzmann-like equations.
Specifically, given an initial probability distribution $\bar \rho_0$ on $\CB(\RE)$, the Borel $\sigma$-field of $\RE$, we consider   a time-dependent
probability measure $\rho_t $ 
solution of the homogeneous kinetic equation 
\begin{equation}
  \label{eq.1}
  \left\{ \begin{aligned}
      & \partial_t \rho_t + \rho_t =
      {Q}^+(\rho_t,\rho_t )
       \\
      & \rho_0=\bar \rho_0.\\
    \end{aligned}
  \right.
\end{equation}
Following
 \cite{BassLadMatth10,CeGaBo}, we assume
that $Q^+$ is the  \textit{smoothing transformation}
defined
by
 \begin{equation}\label{eq.2}
   Q^+(\rho ,\rho)= \text{Law}(LX_1+RX_2) 
 \end{equation}
where  $\rho$ is the law of $X_1,X_2$,
$(L,R)$  is a given random vector of $\RE^2$,  
and $(L,R),X_1$ and $X_2$ are stochastically independent.

The first model of  type \eqref{eq.1}-\eqref{eq.2} has
been introduced by Kac \cite{Kac}, with  collisional parameters
$L=\sin \tilde \theta$ and $R=\cos  \tilde \theta$
{for a random angle $ \tilde \theta$ uniformly distributed  on
$[0,2\pi)$}. In the original Kac equation  
$\rho_t$ represents the probability distribution of the velocity of a particle 
in a homogeneous gas. 
In addition to  the Kac equation, 
also some  one dimensional dissipative Maxwell models for colliding molecules, see e.g.  
\cite{Ben-Avraham,PareschiToscani,PulvirentiToscani}, can be seen as special cases of  \eqref{eq.1}-\eqref{eq.2}. 
Moreover, equations \eqref{eq.1}-\eqref{eq.2} 
have been used  to describe socio-economical dynamics see, e.g.,  \cite{BaLaTo,BCC,DuffieGiroux,MatthesToscani,Pat,Yakovenko} and the references therein. In this last case
particles are replaced by agents in a market and velocities by some quantities of interest (money, wealth, information,...).
Finally, it is worth recalling that, using  results in \cite{CeGaBoBis,CeGaBo},  it can be shown 
that  the  isotropic solution of the multidimensional  inelastic Boltzmann equation \cite{BoCe} can be expressed in terms of the solution
of equation  \eqref{eq.1} for a suitable choice of $(L,R)$. 

The {\it generalized Kac-equation}  \eqref{eq.1}-\eqref{eq.2}
has been extensively studied in many aspects. In particular, the asymptotic behavior of the solutions
of  \eqref{eq.1}-\eqref{eq.2}  has been  treated in details in \cite{BaLa,BassLadMatth10,CeGaBo}.

As for the speed of convergence to equilibrium, explicit rates with respect to suitable probability metrics have been derived 
in various papers. For the Kac equation see \cite{dolera,dolera2,GabettaRegazziniWM},
for the  inelastic Kac equation see \cite{BaLaRe}, for the solutions of the general model  \eqref{eq.1}-\eqref{eq.2} 
see  \cite{BaLa,BassLadMatth10,BaPe}.

Many of  the above mentioned results are based on a probabilistic representation of the solution $\rho_t$. 
In point of fact, as we will briefly explain in Section \ref{S:probrep}, it can be proved that 
the unique solution  $\rho_t$ of \eqref{eq.1}-\eqref{eq.2}
is the law of the stochastic process
\begin{equation}\label{Vt}
V_t=\sum_{j=1}^{\nu_t} \beta_{j\nu_t} X_j 
\end{equation}
where $\nu_t$ is a Yule process, $[\beta_{jn}]_{jn}$ are suitable random weights and $X_j$ are independent identically distributed (i.i.d., for short) random variables with law $\bar \rho_0$. 

The aim of this paper is to study large deviations for
the (eventually rescaled) solution $\rho_t$ when the initial condition $\bar{\rho}_0$ belongs to the domain of normal attraction of an $\alpha$-stable law.
More precisely, we will study the large deviation probability for $e^{-t \mu(\a)}V_t$ 
when, for a suitable $\mu(\a)$, $e^{-t \mu(\a)}V_t$ converges in law to 
a scale mixture of $\a$-stable distributions.      
In the following we shall assume that $\a<2$, the study of the  case $\a=2$ is postponed 
to future work since it  requires completely different techniques. 

In view of the probabilistic representation \eqref{Vt} it is not surprising that the study of the large deviation probabilities for $\rho_t$ is strictly related to large deviations for sums of i.i.d. random variables. 

Let us briefly recall these classical results. 
If $\a \in (0,1) \cup (1,2)$  and if $(X_n)_{n \geq 1}$ is a sequence of i.i.d. random variables
in the domain of normal  attraction of an $\a$-stable law, 
centered if $\a>1$, then, 
$n^{-1/\a} \sum_{i=1}^n X_i$ converges in law to an $\a$-stable random variable. Moreover, if $x_n \to +\infty$,
then 
\begin{equation}\label{largeiid3}
P\Big \{\Big|{n^{-\frac{1}{\a}}} \sum_{i=1}^n X_i\Big|>x_n\Big\} 
\sim {P\{n^{-\frac{1}{\alpha}}\max_{j=1,\dots,n}|X_j|> x_n\}} 
 \sim \frac{c_0}{x_n^\a},
\end{equation}
where $c_0$ is a positive constant determined by the law of $X_1$. See \cite{Heyde1, Heyde2,Heyde4}. For more information on large deviations for sums of i.i.d. random variables
see, for example, \cite{DemboZeit,Vinogradov} and the references therein.

Our main result, which is stated in Theorem \ref{main.thm}, is reminiscent of \eqref{largeiid3}. It can
be summarized by saying that {\it if the initial distribution $\bar \rho_0$ belongs 
to the domain of normal attraction of an $\a$-stable law with $\a<2$ and  the collision coefficients $(L,R)$ satisfy some additional assumptions,
then
\[
 P\{| e^{-t \mu(\a)} V_t|> x_t\} \sim P\{e^{-t \mu(\a)} \max_{j=1,\dots,\nu_t}|\b_{j\nu_t}X_j| > x_t\} \sim\frac{c_0}{x_t^\a}
\]
as $x_t$ goes to $+\infty$}.
As in the i.i.d. case, this result can be interpreted 
by saying that the main part of probability of large deviations is generated by one
large summand comparable with the whole sum process $V_t$. 

The paper is organized as follows. 
Section \ref{sct.weak} is devoted to a brief review of some known results on 
the self-similar asymptotics for the solutions of \eqref{eq.1}. Section \ref{S:probrep} contains the detailed description of the probabilistic representation \eqref{Vt}. In Section \ref{sct.max} we provide some results on the process 
 $H_t=\max_{j=1,\dots,\nu_t}|\b_{j\nu_t}X_j|$. In particular we show that the law of $H_t$ satisfies a kinetic equation of type \eqref{eq.1} for a suitable collisional kernel.
Section \ref{S:mainresults} contains the large deviation results for $\rho_t$. 
Section \ref{S:weightedsums} deals with the study of large deviation probabilities for 
weighted sums of i.i.d. random variables. The proofs of the results stated in Section \ref{S:preliminaries} and \ref{S:mainresults} are collected in Section \ref{sec:proof}.

\section{Self-similar asymptotics for the solutions}\label{S:preliminaries}

In the following, all the random elements are defined on a given probability space $(\Omega,\CF,P)$ and $\E$  denotes the expectation 
with respect to $P$.

Throughout the paper we assume that 
\[ 
\text{\it $L$ and $R$ are non-negative random variables such that
\( P\{L>0\}+P\{R>0\}>1 \).
}
\]

As for the initial probability distribution $\bar \rho_0$ is concerned, we will assume that it belongs
to the {\it domain of normal attraction} of an $\alpha$-stable law. 
It is well-known  that, provided $\a\not=2$,
a probability measure $\bar \rho_0$ belongs to the { domain of normal attraction} of an $\a$-stable law
if and only if its distribution function $F_0(x):=\bar \rho_0\{(-\infty,x]\}$ satisfies
\begin{equation}
  \label{stabledomain}
    \lim_{x \to +\infty} x^\a (1-F_0(x)) =c_0^+<+\infty, \quad
    \lim_{x \to -\infty} |x|^\a F_0(x) =c_0^-<+\infty. 
\end{equation}
Typically, one also requires that $c_0^++c_0^->0$.
See for example  Chapter 2 of \cite{Ibragimov}.

Finally, let us introduce the convex function $\QQ:[0,+\infty)\to[-1,+\infty]$ by
\begin{equation*}
  \QQ(s):=\E[L^s+R^s]-1,
\end{equation*}
with the convention that $0^0=0$ and let 
\[
\qq(s):=\frac{\QQ(s)}{s} \qquad (s>0)
\]
be the so called spectral function of $Q^+$, see \cite{BaLa} and \cite{CeGaBo}.

\subsection{Convergence to self-similar solutions}
\label{sct.weak}

In the study of the asymptotic behavior of the solutions of \eqref{eq.1}, a fundamental role is played by the fixed point equation for distributions 
\begin{equation}\label{eq.fix.mix}
Z \stackrel{\CL}{=} \Theta^{\CS(\a)}(L^\a Z_1 + R^\a Z_2)
\end{equation}
where $Z,Z_1,Z_2$ are i.i.d. positive random variables, $\Theta$ is a random variable with uniform distribution on $(0,1)$, 
$(Z,Z_1,Z_2)$, $\Theta$ and $(L,R)$ are stochastically independent.

As already recalled in the introduction, the unique solution $\rho_t$ to \eqref{eq.1}-\eqref{eq.2}
is the law of the stochastic process $V_t$ defined in \eqref{Vt}. 
Further details on this probabilistic representation  will be given in Section \ref{S:probrep}. 
The next results, concerning the  convergence of a suitable rescaling of $V_t$ to the  
so-called self-similar solutions of \eqref{eq.1}, 
are proved in \cite{BaLa}.

\begin{theorem}[CLT when $\a\not =1$, \cite{BaLa}]\label{thm2} 
Let $\a \in (0,1) \cup (1,2)$ and let condition \eqref{stabledomain} be satisfied for some $(c_0^+,c_0^-)$ such that $c_0^++c_0^->0$, 
with $\int v \bar \rho_0(dv)=0$ if $\a>1$.
If $\qq(\delta)<\qq(\a)<+\infty$ for some $\delta>\alpha$, then $e^{-\qq(\a)t} V_t$ converges in distribution, as $t\to +\infty$, to a random variable $V_\infty$ 
  with the following characteristic function:
  \begin{equation}
    \label{characteristic}
     \E[e^{i \xi  V_\infty} ] 
    =\E[ \exp\{ -|\xi|^\a \lm Z_{\infty}{(\a)} (1-i \eta \tan(\pi\a/2)
    \operatorname{sign}\xi)   \} ]  \qquad (\xi \in \RE)
  \end{equation}
  where 
\begin{equation}
  \label{constant}
  \lm =  \frac{(c_0^{+}+c_0^{-})\pi}{2\Gamma(\a)\sin(\pi\a/2)},
  \qquad \eta = \frac{c_0^{+}-c_0^{-}  }{c_0^{+}+c_0^{-}}
\end{equation} 
and the law of $Z_\infty(\a)$ is the unique positive solution to \eqref{eq.fix.mix} with $\E[Z_\infty{(\a)}]=1$. 
\end{theorem}

Further information on the mixing random variable $Z_\infty(\a)$  are given in Proposition \ref{Lemma7new}. See also \cite{BaLa}.

The results concerning the case $\alpha=1$ are here stated under slightly more general assumptions than 
in \cite{BaLa}. For completeness a sketch of the proof is given 
in Section \ref{sec:proof}. 

\begin{theorem}[CLT when $\alpha=1$]\label{thm1}
Let \eqref{stabledomain}  holds with $\a=1$ and $c_0^+=c_0^->0$.
Suppose, in addition, that
\begin{equation}\label{gamma}
\lim_{R\to+\infty}\int_{(-R,R)}x dF_0(x)=\gamma_0
\end{equation}
with $-\infty<\gamma_0<+\infty$.
If $\qq(\delta)<\qq(1)<+\infty$ for some $\delta>1$, then $e^{-\qq(1)t} V_t$ converges in distribution, as $t\to +\infty$, to a random variable $V_\infty$   with the following characteristic function:
\begin{equation}\label{characteristic2}
\E[\exp(i \xi  V_\infty) ] =\E[ e^{ Z_{\infty}{(1)} (i  \gamma_0 \xi- c_0^+\pi|\xi|)} ]
\end{equation}
and the law of $Z_\infty(1)$ is the unique positive solution to \eqref{eq.fix.mix} for $\a=1$, with $\E[Z_\infty{(1)}]=1$.
\end{theorem}

\begin{remark} In order to study the large deviations for $\rho_t$, in what follows we will need to assume that $c_0^++c_0^->0$, even if 
both {\rm Theorem \ref{thm2}} and {\rm Theorem \ref{thm1}} hold also for $c_0^++c_0^-=0$.
In this last case, {\rm Theorem \ref{thm2}} is valid with $\lambda=\eta=0$ and hence $V_\infty=0$ with probability one,
while  {\rm Theorem \ref{thm1}} is valid
with $V_\infty=\gamma_0 Z_\infty(1)$. 
\end{remark}

\begin{remark}
Let us consider a random vector $(L,R)$  such that $\qq(\a)=0$, that is $\E[L^\a+R^\a]=1$.
As a consequence of the previous results, if $\E[L^\delta+R^\delta]<1$ for some $\delta>\a$, then 
$V_t$ converges in  distribution  to $V_\infty$. In this case $Z_\infty(\alpha)$ satisfies 
the fixed point equation 
\begin{equation*}%
Z \stackrel{\CL}{=} L^\a Z_1 + R^\a Z_2
\end{equation*}
and it is easy to see that the law $\rho_\infty$  of $V_\infty$ is a steady state for equation \eqref{eq.1}, i.e.
$\rho_\infty=Q^+(\rho_\infty,\rho_\infty)$.  This case has been extensively studied in \cite{BassLadMatth10}.
\end{remark}

\subsection{Probabilistic representation of the solution}\label{S:probrep}
The proofs of Theorems \ref{thm2} and \ref{thm1} 
are based on the fact that $V_t$ is a randomly weighted sum of i.i.d. random variables. In \cite{BassLadMatth10} it has been shown that
the unique solution of  \eqref{eq.1}-\eqref{eq.2}  with initial datum $\bar{\rho}_0$ is the law of
\begin{equation*}
V_t=\sum_{j=1}^{\nu_t}\beta_{j,\nu_t}X_j,
\end{equation*}
provided that
\begin{itemize}
\item $(X_j)_{j\geq1}$ is a sequence of i.i.d. random variables with distribution $\bar\rho_0$;
\item $(\nu_t)_{t\geq0}$ is a Yule process, see e.g. \cite{AthNey}, hence in particular 
        \[
        P\{\nu_t=n\}=e^{-t}(1-e^{-t})^{n-1}
        \]
        for every $n\geq1$ and $t\geq0$;
\item  $(\beta_{j,n}:\;j=1,\dots,n)_{n\geq1}$  is an array of non-negative random weights;
\item  $(X_j)_{j\geq1}$, $(\nu_t)_{t\geq0}$ and $(\beta_{j,n}:\;j=1,\dots,n)_{n\geq1}$ 
are stochastically independent. 
\end{itemize}
As to the definition of the weights $\b_{jn}$'s is concerned: 
$\beta_{1,1}:=1$, $(\beta_{1,2},\beta_{2,2}):=(L_1,R_1)$ and, for any $n\geq 2$,
\begin{equation}
  \label{recursion}
  \begin{split}
    (\beta_{1,n+1}, & \ldots,\beta_{n+1,n+1})  
    \\& := (\beta_{1,n},\ldots,\beta_{I_n-1,n}, L_n \beta_{I_n,n}, R_n \beta_{I_n,n},\beta_{I_n+1,n},\ldots, \beta_{n,n}), \\
  \end{split}
\end{equation}
where $(L_n,R_n)_{n\geq1}$ is a sequence of i.i.d. random vectors distributed as $(L,R)$,
$(I_n)_{n\geq1}$ is a sequence of independent random variables uniformly distributed on $\{1,\dots,n\}$ for every $n\geq1$, 
$(L_n,R_n)_{n\geq1}$ and $(I_n)_{n\geq1}$ are independent.

\subsection{The max-process $\frH_t$}\label{sct.max}
Since we shall compare the large deviations
of  $e^{-\qq(\a)t}{V}_t$ with the large deviations of 
$e^{-\qq(\a)t} \frH_t$, where 
\[
\frH_t=\max_{1 \leq j \leq \nu_t} | \beta_{j,\nu_t} X_j|,
\]
we start by providing some results on this last process. First of all, it is worth noticing that the law of $\frH_t$ satisfies an homogeneous kinetic equation of the form 
\eqref{eq.1} with $Q^+$ replaced by
the kernel
 \begin{equation}\label{eq.2maxversion}
   \tilde Q^+(\rho ,\rho)= \text{Law}(\max\{L|X_1|,R|X_2|\})
 \end{equation}
where, as usual, $X_1,X_2,(L,R)$ are independent and $X_i$ has law $\rho$ for $i=1,2$. 
\begin{theorem}\label{thm.eqdiff}
 Let ${\fHH}_t(x):=P\{ \frH_t \leq x\}$, then
 \begin{equation}
  \label{eq.max}
  \left\{ \begin{aligned}
      & \partial_t {\fHH}_t(x) + {\fHH}_t(x) =
      \E[{\fHH}_t({x}/{L}) {\fHH}_t(x/R)]
       \\
      & {\fHH}_0(x)= P\{ |X_1| \leq x\}.\\
    \end{aligned}
  \right.
\end{equation}
for every $x$ in $\RE$, with the convention ${\fHH}_t({x}/{0})= 0$ if $x<0$ and ${\fHH}_t({x}/{0})= 1$ otherwise. 
\end{theorem}

Following the same line of reasoning of \cite{BaLa,BassLadMatth10} we prove the next result on the asymptotic 
behavior of $e^{-\qq(\a)t} \frH_t$.

\begin{theorem}\label{maxlimittheorem}
 Let $\a \in (0,1) \cup (1,2)$ and the hypotheses of {\rm Theorem \ref{thm2}} be in force, 
or let $\a=1$ and the hypotheses of {\rm Theorem \ref{thm1}} hold.  Assume also that  $c_0=c_0^++c_0^->0$.
Then $e^{-\qq(\a)t} \frH_t$ converges in distribution, as $t\to +\infty$, to a random variable $\frH_\infty$ 
with the following probability distribution function:
  \begin{equation}
    \label{maxlimit}
 P\{ \frH_\infty \leq x \} =
\left\{ \begin{array}{ll}
       \E\Big[ e^{-\frac{c_0}{|x|^{\a}} Z_{\infty}{(\a)}} \Big] & \text{if $x > 0$}
       \\
P\{ Z_{\infty}{(\a)}=0\}  & \text{if $x = 0$}\\
      0 & \text{if $x < 0$ }
    \end{array}
  \right.
 \end{equation}
  where  the law of $Z_\infty(\a)$ is the unique positive solution to \eqref{eq.fix.mix} with $\E[Z_\infty{(\a)}]=1$. 
\end{theorem}

It is useful to note that Theorem \ref{maxlimittheorem} states that the law of $\frH_\infty$ is a scale mixture of Fr\'echet distributions.

\section{Main results: large deviations for $\rho_t$}\label{S:mainresults} 

As a consequence of Theorems \ref{thm2}-\ref{thm1}, one has that, if $x_t \to +\infty$ as $t \to +\infty$, then 
\[
\lim_{t \to +\infty} P\{|e^{-\qq(\a)t}V_t| > x_t\} =0.
\]
The main result of this paper concerns the study of the speed of convergence of such a probability to zero under suitable conditions
on the function $\qq(s)$.
In order to state the results, we need some more notation. 
When $\CS(2\a)<+\infty$ let $h(t):[0,+\infty) \to [0,+\infty)$ be the function  
\begin{equation}\label{defhA}
h(t):=
\left\{%
\begin{array}{ll}
t & \text{if $\qq(2\a)<\qq(\a)$ and $2\CS(\a)=-1$;} \\
e^{-(2\CS(\a)+1)t} & \text{if $\qq(2\a)<\qq(\a)$ and $2\CS(\a)<-1$;} \\
e^{2\a(\qq(2\a)-\qq(\a))t} & \text{if $\qq(2\a)>\qq(\a)$;} \\
e^{\eta t} & \text{if $\qq(2\a)=\qq(\a)$ and $0<\CS(\a)$ for a fixed $\eta>0$;} \\
te^{-(2\CS(\a)+1)t} & \text{if $\qq(2\a)=\qq(\a)$ and $2\CS(\a)<-1$;} \\
t^2 & \text{if $\qq(2\a)=\qq(\a)$ and $2\CS(\a)=-1$;} \\
t & \text{if $\qq(2\a)=\qq(\a)$ and $-1<2\CS(\a) \leq 0$.} \\
\end{array}
\right.
\end{equation}

\begin{theorem}[Large deviations]\label{main.thm}
Let $\a \in (0,1) \cup (1,2)$ and the hypotheses of {\rm Theorem \ref{thm2}} be in force, 
or let $\a=1$ and the hypotheses of {\rm Theorem \ref{thm1}} hold.  Assume also that $\CS(2\a)<+\infty$ and $c_0:=c_0^++c_0^->0$.
\begin{itemize}
 \item If $\qq(2\a)<\qq(\a)$ and $2\CS(\a)>-1$, then, for every $x_t$ such that $x_t \to +\infty$ as $t \to +\infty$ , one has
\begin{equation}\label{eq.7}
\lim_{t \to +\infty} \frac{x_t^\alpha}{c_0} P\{|e^{-\qq(\a)t}V_t| > x_t\} = \lim_{t \to +\infty} \frac{P\{|e^{-\qq(\a)t}V_t| > x_t\}}{P\{|V_\infty| > x_t\}}=1
\end{equation}
and
\begin{equation}\label{eq.7tre}
\lim_{t \to +\infty} \frac{P\{|e^{-\qq(\a)t}V_t| > x_t\}}{P\{|e^{-\qq(\a)t}\frH_t| > x_t\}}=1.
\end{equation}
\item  If either  $\qq(2\a) \geq \qq(\a)$ or  $2\CS(\a)\leq -1$  and $x_t$ is such that $x_t^{\a-\eps}/h(t) \to +\infty$ as $t \to +\infty$ 
for some $\eps>0$, with $h(t)$ as in \eqref{defhA}, then
\eqref{eq.7}-\eqref{eq.7tre} hold true.
\end{itemize}
\end{theorem}

\begin{remark}
Let us consider {\rm Theorem \ref{main.thm}} in the particular case in which $\E[L^\a+R^\a]=1$ and hence $0=2\CS (\a)> -1$. Then, if $\E[L^{2\a}+R^{2\a}]<1$ and  $x_t \to +\infty$ as $t \to +\infty$,  one has
\begin{equation}\label{eq.7bis}
\lim_{t \to +\infty} \frac{x_t^\alpha}{c_0} P\{|V_t| > x_t\} = \lim_{t \to +\infty} \frac{P\{|V_t| > x_t\}}{P\{|V_\infty| > x_t\}}
= \lim_{t \to +\infty} \frac{P\{|V_t| > x_t\}}{P\{|\frH_t| > x_t\}}=1
\end{equation}
where the law of  $V_\infty$ is a steady state for equation \eqref{eq.1}.
\end{remark}

As pointed out in the Introduction, the results stated in the previous theorem are related to large deviations for sums of i.i.d. random variables:
{\it Let  $\a \in (0,1) \cup (1,2)$  and let $(X_n)_{n \geq 1}$ be a sequence of i.i.d. random variables
in the domain of normal  attraction of an $\a$-stable law, 
centered for $\a>1$, then,  
\begin{equation}\label{largeiid}
 \lim_{n \to + \infty} \frac{P\Big \{\Big|{n^{-\frac{1}{\a}}} \sum_{i=1}^n X_i\Big|>x_n\Big\}}{n  P\{|X_1|>n^{1/\alpha} x_n\}} =
 \lim_{n \to + \infty} \frac{P\Big \{\Big|{n^{-\frac{1}{\a}}} \sum_{i=1}^n X_i\Big|>x_n\Big\}}{P\{\max_{j=1,\dots,n}|X_j|>n^{1/\alpha} x_n\}} 
 =1
\end{equation}
whenever $x_n \to +\infty$. 
}
See \cite{Heyde2} and \cite{Heyde4}.
It follows from 
\eqref{stabledomain} that  
$P\{|X_1|> n^{\frac{1}{\a}}x_n\}  \sim c_0/(nx_n^\a)$.
Moreover, if $S_\a$ is the $\a$-stable random variable limit of $n^{-\frac{1}{\a}} \sum_{i=1}^n X_i $,
then, $P\{|S_\a|> x_n\} \sim c_0/x_n^\a$, since each stable random variable belongs to its own domain of normal attraction. Consequently
\begin{equation}\label{largeiid2}
 \lim_{n \to + \infty} \frac{P\Big \{\Big|{n^{-\frac{1}{\a}}} \sum_{i=1}^n X_i\Big|>x_n\Big\}}{P\{|S_\a|> x_n\}}=
 \lim_{n \to + \infty} \frac{x_n^\a}{c_0}P\Big \{\Big|{n^{-\frac{1}{\a}}} \sum_{i=1}^n X_i\Big|>x_n\Big\} =1.
\end{equation}

At this stage, it should be clear that equations \eqref{eq.7}-\eqref{eq.7tre}-\eqref{eq.7bis} provide 
analogous results for our processes.

\section{Large deviation for sum of weighted i.i.d. random variables}\label{S:weightedsums}
The present section deals with the study of the probability of large deviations for 
 weighted sums of i.i.d. random variables. This study is a  generalization of the large deviation estimates presented in \cite{Heyde2,Heyde4} and, besides the interest 
 it could hold in itself, it is the first step in the proof of Theorem \ref{main.thm}.

Let $(X_j)_{j \geq 1}$ be a sequence of i.i.d. random variables with common distirbution function $F_0$
and $[b_{jn}: j=1,\dots,n; n \geq 1]$ be an array of non-negative weights. Let
\[
 S_{n}:=\sum_{j=1}^n b_{jn} X_j,
\]
$b(n):=\max\{b_{jn}: j=1,\dots,n\}$ and $b^{(1:n)}:=(b_{1n},\dots,b_{nn})$.

If $F_0$ satisfies \eqref{stabledomain}, for every $x>0$ define
\[
 \begin{split}
& \CR(x):= \frac{x^\alpha}{c_0} P\{|X_1|>x\}-1  \qquad (c_0:=c_0^+ + c_0^-)  \\
& \bar \CR(x):=\sup_{y:y\geq x} |\CR(x)|. \\
\end{split}
\]
Clearly 
\begin{equation}\label{eq.8}
P\{|X_1|>x\}=c_0x^{-\alpha}(1+\CR(x)),
\end{equation}
hence $\|\CR\|_\infty:=\sup_{x>0}|\CR(x)|<+\infty$
and
\begin{equation}\label{eq.8bis}
 \lim_{x \to +\infty} \bar \CR(x)=\lim_{x \to +\infty} \CR(x)=0.
\end{equation}
Finally, set 
\begin{equation}\label{eq.8_3}
K_0:=c_0(\|\CR\|_\infty+1)
\end{equation}
and
\begin{equation*}
 \Delta_{b^{(1:n)}}^{(n)}(y):=P\{|S_n|+ b(n)|X_1| \leq y  \}.
\end{equation*}

\begin{lemma}\label{lemma4.1} Assume that $F_0$ satisfies \eqref{stabledomain}  with $c_0=c_0^++c_0^->0$. 
Moreover, if $\a=1$ assume that $c_0^+=c_0^-$ and that \eqref{gamma} holds, while  if $\a>1$
 assume that $\E[X_1]=0$. Then, for every $x>0$, $n \geq 1$, $0<\eps<1$ and $\gamma>0$, 
the following inequalities are valid
\begin{equation}\label{lowerbound}
\begin{split}
x^\alpha P\{|S_n|>x\} \geq& 
\frac{\Delta_{b^{(1:n)}}^{(n)}(\eps x)}{(1+\eps)^\alpha} c_0
\left (1-\bar \CR\left( \frac{x(1+\eps)}{b(n)}\right)\right) \sum_{j=1}^n b_{jn}^\alpha 
\\ 
&  
-\frac{ K_0 ^2}{x^\alpha(1+\eps)^{2\alpha}} \left (\sum_{j=1}^nb_{jn}^\alpha \right )^2  \\
\end{split}
\end{equation}
and
\begin{equation}\label{upperbound}
\begin{split}
x^\alpha P\{|S_n|>x\} & \leq 
   \Big [ \frac{c_0 }{(1-\eps)^\alpha} \left (1+\bar \CR \left(\frac{x(1-\eps)}{b(n)}\right)\right) 
+\frac{2K_0}{\eps^2(2-\alpha) x^{(2-\alpha)(1-\gamma)}} \Big]   \sum_{j=1}^n b_{jn}^\alpha
\\ 
&  
+   \left [ \frac{K_0^2}{x^{\alpha(2\gamma-1)}} +
\frac{K_1}{\eps^2 x^{2-\alpha+2(\alpha-1)\gamma}} \right]  \left ( \sum_{j=1}^nb_{jn}^\alpha \right )^2
\\
\end{split}
\end{equation}
where $K_1=K_0^2/(1-\alpha)^2$ if $\alpha<1$, $K_1=K_0^2\alpha^2/(1-\alpha)^2$ if $\a>1$ and $K_1=(\gamma_0+\sup_R|\int_{(-R,R)}y dF_0(y)-\gamma_0|)^2$ if $\a=1$. 
Moreover,
\begin{equation}\label{maxineq}
\begin{split}
c_0  \sum_{j=1}^{n}  b^\a_{jn} 
\Big(1-\bar \CR\Big(\frac{x}{b(n)}\Big)\Big) 
&-\frac{K_0^2}{x^{\a}}   \left( \sum_{j=1}^{n} b^\a_{jn}  \right)^2 
\leq x^\a P\{\max_{1 \leq j \leq  n}|b_{jn} X_j | >x \} \\ 
& \leq  c_0  \sum_{j=1}^{n}  b^\a_{jn} 
\Big(1+\bar \CR\Big(\frac{x}{b(n)}\Big)\Big). 
\\
\end{split}
\end{equation}
\end{lemma}

\begin{proof}
The proof of this lemma is an adaptation to the present case of the techniques used in \cite{Heyde1,Heyde2}.

{\underline{Proof of \eqref{lowerbound}}.} 
Set
\[
 S_{n,k}:= \sum_{1 \leq j \leq n, j \not=k} b_{jn} X_j \qquad k=1,\dots,n
\]
and
\[
A_{j}:=\{ |b_{jn}X_j| > (1+\eps) x, |S_{n,j}| \leq \eps x\}.
\]
 Clearly 
\[
\cup_{j=1}^n A_j \subset \{|S_n| >x \} 
\]
and hence, by Bonferroni inequality, 
\[
P\{|S_n| >x \} \geq \sum_{j=1}^n P(A_j)- \sum_{1 \leq j < k \leq n} P(A_j \cap A_k).
\]
Now, from the independence of the $X_j$'s, one obtains
\[
P(A_j \cap A_k) \leq P\{ |b_{jn}X_j|>(1+\eps)x\}P\{ |b_{kn}X_k|>(1+\eps)x\}.
\]
and
\[
P(A_j) = P\{ |b_{jn}X_j|>(1+\eps)x\}P\{ |S_{n,j}|\leq \eps x\}.
\]
Hence
\begin{equation}\label{1est}
P\{|S_n|>x\}  
\geq \sum_{j=1}^n P\{ |b_{jn}X_j|>(1+\eps)x\}P\{ |S_{n,j}|\leq \eps x\}
-\left( \sum_{j=1}^n P\{ |b_{jn}X_j|>(1+\eps)x\} \right)^2.
\end{equation}
Furthermore, for every $j=1,\dots,n$, 
\begin{equation}\label{2est}
P\{ |S_{n,j}|\leq \eps x\} \geq P\{|S_n|+ b(n)|X_j| \leq \eps x \}
= \Delta_{b^{(1:n)}}^{(n)}(y)
\end{equation}
and from \eqref{eq.8}-\eqref{eq.8_3} one gets
\begin{equation}\label{3est}
\frac{c_0b_{jn}^\alpha}{x^\alpha(1+\eps)^\alpha}\Big(1-\bar \CR\Big(\frac{x(1+\eps)}{b(n)}\Big)\Big) \leq 
P\{ |b_{jn}X_j|>(1+\eps)x\}\leq  \frac{b_{jn}^\alpha}{x^\alpha(1+\eps)^\alpha} K_0.
\end{equation}
Combining \eqref{1est}, \eqref{2est} and \eqref{3est} one obtains \eqref{lowerbound}.

{\underline{Proof of \eqref{upperbound}}.} 
Define
\[
\begin{split}
& Y_{jn}:=b_{jn} X_j \J\{|b_{jn}X_j| \leq x^\gamma \} \\
& \tilde S_n := \sum_{j=1}^n Y_{jn} \\
& E_n:= \cup_{j=1}^n \{|b_{jn} X_j|> (1-\eps)x \} \\
& F_n:= \cup_{1 \leq i <j \leq n}  \{|b_{jn} X_j|> x^\gamma,|b_{in} X_i|> x^\gamma \} \\
& G_n:=\{|\tilde S_n| > \eps x \} \\
\end{split}
\]
It is easy to see that $\{ |S_n| > x\} \subset E_n \cup F_n \cup G_n$ and hence,
\begin{equation}\label{eq.15uno}
 P( |S_n| > x) \leq P(E_n)+P(F_n)+P(G_n).
\end{equation}
From \eqref{eq.8} one obtains
\begin{equation}\label{eq.15}
\begin{split}
P(E_n) & \leq \sum_{j=1}^n P(  |b_{jn}X_j| > (1-\eps) x)
 = \sum_{j=1}^n \frac{c_0 b_{jn}^\alpha}{x^\alpha(1-\eps)^\alpha}\Big(1+\CR\Big(\frac{x(1-\eps)}{b_{jn}}\Big)\Big) \\
  &\leq  \sum_{j=1}^n \frac{c_0 b_{jn}^\alpha}{x^\alpha(1-\eps)^\alpha}\Big(1+\bar \CR\Big(\frac{x(1-\eps)}{b(n)}\Big)\Big)
\end{split}
\end{equation}
and
\begin{equation}\label{eq.16}
\begin{split}
P(F_n) & \leq \sum_{1 \leq i <j \leq n} P(  |b_{in}X_i| > x^\gamma ) P(  |b_{jn}X_j| > x^\gamma ) \\
&=  \sum_{1 \leq i <j \leq n} \frac{c_0^2 b_{in}^\alpha b_{jn}^\alpha}{x^{2\gamma\alpha}}
\Big(1+\CR\Big(\frac{x^\gamma}{b_{in}}\Big)\Big)\Big(1+\CR\Big(\frac{x^\gamma}{b_{jn}}\Big)\Big)\\
& \leq K_0^2 \Big ( \sum_{j=1}^n b_{jn}^\alpha \Big)^2 x^{-2\a\gamma}
\end{split}
\end{equation}
where $K_0$ is defined in \eqref{eq.8_3} and $\CR(x^\gamma/0):=0$.
From Chebyshev inequality 
\begin{equation}\label{eq.16bis}
\begin{split}
P(G_n) & \leq \frac{1}{\eps^2x^2} \E[\tilde S_n^2]
\leq
\frac{1}{\eps^2x^2} \E\Big[\sum_{j=1}^n Y_{jn}^2 +  \sum_{1\leq i,j \leq n} Y_{in}Y_{jn}\Big]\\
& \leq \frac{1}{\eps^2x^2} \Big (\sum_{j=1}^n \E[ Y_{jn}^2]+ \Big (\sum_{j=1}^n| \E[ Y_{jn}]| \Big)^2 \Big)
 \end{split}
\end{equation}

Note that if $b_{jn}=0$ then
$\E[Y_{jn}^2]=|\E[Y_{jn}]|=0$, hence from now on we assume that $b_{jn}>0$. 
Now
\[
\E[Y_{jn}^2]=b_{jn}^2 \E[|X_{j}|^2 \J\{|X_j|\leq x^\g/b_{jn}\}]
\leq 2 b_{jn}^2 \int_0^{ x^\g/b_{jn}} y P\{|X_1|>y\}dy.
\]
Since $P\{|X_1|>y\}\leq K_0 y^{-\alpha}$, it follows that
\begin{equation}\label{eq.17}
\E[Y_{jn}^2] \leq \frac{2K_0}{2-\alpha} b_{jn}^\alpha x^{(2-\a)\gamma}.  
\end{equation}
It remains to consider $|\E[Y_{jn}]|$. If $\alpha<1$,
then
\begin{equation}\label{eq.18}
|\E[Y_{jn}]| \leq b_{jn} \int_{0}^{x^\g/b_{jn}} P\{|X_1|>y\} dy  \leq b_{jn} K_0 \int_{0}^{x^\g/b_{jn}} y^{-\a}dy
= \frac{ b^\a_{jn} K_0}{1-\alpha} x^{(1-\a)\gamma}.
\end{equation}
If $\a>1$ and $\E[X_1]=\int y dF_0(y)=0$, then
\begin{equation}\label{eq.19}
\begin{split}
|\E[Y_{jn}]| &=b_{jn} \Big |\int_{\{y: |y| \leq x^\g/b_{jn}\}}  ydF_0(y) \Big |
=b_{jn} \Big | \int_{\{y: |y| > x^\g/b_{jn}\}} y  dF_0(y)\Big  |\\
& \leq b_{jn} \Big [ \int_{x^\g/b_{jn}}^{+\infty} P\{|X_1|>y\} dy
 +{\frac{x^\g}{b_{jn}}} P\Big\{|X_1|>\frac{x^\g}{b_{jn}}\Big\} \Big]\\
& 
\leq b_{jn} K_0 \Big [ \int_{x^\g/b_{jn}}^{+\infty}y^{-\a} dy
 +{x^{\g(1-\a)} b_{jn}^{\a-1}} \Big]= b_{jn}^\a  K_0  \frac{\a}{\a-1}x^{(1-\a)\gamma}.
\\
\end{split}
\end{equation}
Finally, if $\a=1$, by assumption
\[
K:=\sup_{R>0} |\int_{(-R,R)} yF_0(y)-\gamma_0|<+\infty.
\]
Hence, in this case, one gets
\begin{equation}\label{eq.19BIS}
|\E[Y_{jn}]| \leq b_{jn} \Big |\int_{\{y: |y| \leq x^\g/b_{jn}\}}  ydF_0(y) -\gamma_0\Big |+b_{jn}  \gamma_0 
\leq  b_{jn} (\gamma_0+K).
\end{equation}
Combining \eqref{eq.15uno}-\eqref{eq.19BIS} one gets \eqref{upperbound}.

{\underline{Proof of \eqref{maxineq}}.} 
By Bonferroni inequality, using once again \eqref{eq.8} and \eqref{eq.8_3}, one gets
\[
\begin{split}
P\{\max_{1 \leq j \leq  n}|b_{jn} X_j | >x \} &\geq 
\sum_{j=1}^{n} P\{|b_{jn} X_j | >x \} 
- \sum_{1 \leq j < k \leq n} P\{ |b_{jn} X_j | >x, |b_{kn} X_k | >x  \}\\
&\geq \frac{c_0 }{x^\a} \sum_{j=1}^{n}  b^\a_{jn} 
\Big(1-\bar \CR\Big(\frac{x}{b(n)}\Big)\Big) 
-\frac{K_0^2}{x^{2\a}}   \Big( \sum_{j=1}^{n} b^\a_{jn}  \Big)^2 
\\
\end{split}
\]
and
\[
P\{\max_{1 \leq j \leq  n}|b_{j,n} X_j | >x \} \leq 
\sum_{j=1}^{n} P\{|b_{jn} X_j | >x \} \leq
\frac{c_0 }{x^\a} \sum_{j=1}^{n}  b^\a_{jn} 
\Big(1+\bar \CR\Big(\frac{x}{b(n)}\Big)\Big) 
\]
that yields \eqref{maxineq}.
\end{proof} 

\begin{remark}\label{remark1}
Notice that if $\gamma \in (1/2,1)$ and $\a \in (0,2)$, then  $(2-\a)(1-\gamma)>0$ and $\a(2\gamma-1)>0$. Moreover, 
if $\gamma<1$ and $\a<1$, then $(2-\a)+2(\a-1)\gamma>\a>0$, while, if $\a>1$, then
$(2-\a)+2(\a-1)\gamma \uparrow \a$ for $\gamma \uparrow 1$. Finally, $\a(2\gamma-1)\uparrow \a$ when $\gamma \uparrow 1$.
\end{remark}
 
 A simple consequence of Lemma \ref{lemma4.1} and Remark \ref{remark1} is the following large deviations result for
the weighted sum  $S_n=\sum_{j=1}^n b_{jn} X_j$.

\begin{corollary} Assume that $F_0$ satisfies \eqref{stabledomain}  with $c_0=c_0^++c_0^->0$. 
If $\a=1$ assume also that $c_0^+=c_0^-$ and that \eqref{gamma} holds, while  if $\a>1$
 assume that $\E[X_1]=0$. If $b(n) \to 0$, $\sum_{j=1}^{n} b_{jn}^\a \to 1$ 
and $x_n \to +\infty$, then
\[
\lim_{n \to +\infty} {x_n^\alpha} P\{|S_n| > x_n\} =c_0 .
\]
\end{corollary}

\section{Proofs}\label{sec:proof}

\subsection{Preliminary results}

Let $\a$ be a given positive real number such that
$\E[ L^\a+R^\a]<+\infty$.
For every integer number $n \geq 1$ set
\begin{align}
  \label{eq.them}
  M_n{(\a)} := \sum_{j=1}^n \beta_{j,n}^\a \qquad \text{and} \qquad \tilde M_n (\a):= \frac{M_n{(\a)}}{m_n(\a)}
\end{align}
where
\[
m_n(\a):= \frac{\Gamma(n+\CS(\a))}{\Gamma(n)\Gamma(\CS(\a)+1)}.
\]
Note that, as $n \to +\infty$, by the well-known asymptotic expansion for the ratio of Gamma functions, 
\begin{equation}\label{asintm}
m_n(\a)=n^{\QQ(\a)} \frac{1}{\Gamma(\QQ(\a)+1)} 
\Big(1+O\Big(\frac{1}{n}\Big)\Big).
\end{equation}
For every $\a>0$, set also
\[
 \b_{(n)}:=\max_{1\leq j \leq n} \beta_{j,n} \qquad \text{and} \qquad \tilde \b_{(n)}:= \frac{\b(n)  }{m_n(\a)^{\frac{1}{\a}}},
  \]
and recall that
$\qq(\a)=\QQ(\a)/\a$.
Let us collect some results related to the sequence $(\tilde M_n(\a))_{n\geq 1}$ proved in \cite{BaLa}.  

\begin{proposition}[\cite{BaLa}]\label{Lemma2BL}
Let $\a>0$ such that $\E[ L^\a+R^\a] <+\infty$.
\begin{itemize}
 \item[(i)] For every $n\geq 1$
\[
\E[M_n(\a)]= m_n(\a).
\]
\item[(ii)] $\tilde M_n(\a)$ is a positive martingale with respect to the filtration $(\CG_n)_{n \geq 1}$ with
\[
\CG_n=\s(L_1,R_1,\dots,L_{n-1},R_{n-1},I_1,\dots,I_{n-1}),
\]
and $\E[\tilde M_n(\a)]=1$. Hence, $\tilde M_n(\a)$ converges almost surely to
a random variable $\tilde M_\infty(\a)$ with $\E[\tilde M_\infty(\a)] \leq 1$.
\item[(iii)]  If for some $\delta>0$ and $\a>0$ one has
    $\qq(\delta) < \qq(\a)<+\infty$, then 
    $\tilde \b_{(n)}$  converges in probability to $0$.
\item[(iv)] If  $\qq(\delta) < \qq(\a)<+\infty$ for $\a < \delta$, $\tilde M_n(\a)$ converges in $L^1$ to $\tilde M_\infty(\a)$
and $\E [\tilde M_\infty(\a)]=1$. 
\end{itemize}
\end{proposition}

Let us define, for every $t \geq 0$,
\[
 Y_t:=m_{\nu_t}(\a) e^{-\CS(\a) t}.
\]

\begin{proposition}\label{Lemma7new}
Let  $\qq(\delta) < \qq(\a)<+\infty$ for $\a < \delta$ and let $\tilde M_\infty(\a)$ be the same random variable of {\rm Proposition
\ref{Lemma2BL}}. 
Then, there exists a random variable $\CE$ with exponential distribution of parameter $1$, with $\CE$ and $\tilde M_\infty(\alpha)$ independent, such that

\begin{equation}\label{convyule}
Y_t \to \frac{\CE^{\CS(\a)}}{\Gamma(\CS(\a)+1)} \qquad\qquad \qquad \qquad\text{a.s.}, 
\end{equation}
and
\begin{equation}\label{convL1}
e^{-\CS(\a)t} M_{\nu_t}(\a) \to \frac{\CE^{\CS(\a)} \tilde M_\infty(\a)}{\Gamma(\CS(\a)+1)}=:Z_\infty(\a)  \qquad \text{a.s. and in $L^1$} 
\end{equation}
 as $t \to +\infty$. Moreover, for every $t$,
\begin{equation}\label{media1}
\E[e^{-\CS(\a)t} M_{\nu_t}(\a)]=\E [Z_\infty(\a)]=1,
\end{equation}
the law of $Z_\infty(\a)$ satisfies the fixed point equation \eqref{eq.fix.mix} and
\begin{equation}\label{media-delta}
\E[Z_\infty(\a)^{\delta/\a}] < +\infty.
\end{equation}
 Finally, 
\begin{equation}\label{betatozero}
\tilde \beta_{(\nu_t)} \to 0 \quad \text{and} \quad \beta_{(\nu_t)} e^{-\qq(\a) t} \to 0  
\end{equation}
in probability as $t \to +\infty$. 
\end{proposition}

\begin{proof}
 It is well-known that if $(\nu_t)_{t}$ is a Yule process, then  $e^{-t}\nu_t$ is a martingale and converges a.s. to an exponential random variable
$\CE$ of parameter $1$, see e.g.  \cite{AthNey}. Hence, by \eqref{asintm}, 
$Y_t=e^{-\CS(\a)t} m_{\nu_t}(\alpha)$ converges a.s. to $\CE^{\CS(\a)}/\Gamma(\CS(\a)+1)$.
By  (iv) of Proposition \ref{Lemma2BL}, it follows that
$\tilde M_{\nu_t}(\a)$ converges a.s. and in $L^1$ to $\tilde M_\infty(\a)$.
Note that $\tilde M_{\infty}(\alpha)$ is measurable with respect to the $\s$-field generated by the $\b_{jn}$'s and
$\CE$ is measurable with respect to the $\s$-field generated by $(\nu_t)_{t}$. This implies that  $\CE$ and $\tilde M_\infty(\alpha)$
are independent. 
Since
$e^{-\CS(\a)t} M_{\nu_t}(\a)=Y_t \tilde M_{\nu_t}(\a)$,
it follows that $e^{-\CS(\a)t} M_{\nu_t}(\a)$ converges a.s. to $ \CE^{\CS(\a)} \tilde M_\infty(\a)/\Gamma(\CS(\a)+1)$.
Moreover, recalling that
for every $\gamma>-1$ and $0<u<1$
\begin{equation}\label{serie-gamma}
\sum_{n=1}^{+\infty}\frac{\Gamma (\gamma+n)}{\Gamma (n)\Gamma (\gamma+1)}(1-u)^{n-1}=u^{-(\gamma+1)}
\end{equation}
and in view of (i) of Proposition \ref{Lemma2BL}
\[
\begin{split}
 \E[e^{-\CS(\a)t} M_{\nu_t}(\a)]=& e^{-\CS(\a)t}\sum_{n=1}^{+\infty}e^{-t}(1-e^{-t})^{n-1}m_n(\a) \\
=&e^{-(\CS(\a)+1)t}\sum_{n=1}^{+\infty}(1-e^{-t})^{n-1}\frac{\Gamma (\CS(\a)+n)}{\Gamma (n)\Gamma (\CS(\a)+1)}=1
\end{split}
\]
for every $t$. By the independence of $\CE$ and $\tilde M_\infty(\a)$ and by (iv) of Proposition \ref{Lemma2BL} one easily see that
\[
\E[Z_\infty(\a)]=\E \Big[\tilde M_\infty(\a) \frac{\CE^{\CS(\a)} }{\Gamma(\CS(\a)+1)}\Big]=
\E[\tilde M_\infty(\a)]\E\Big [ \frac{\CE^{\CS(\a)}}{\Gamma(\CS(\a)+1)}\Big]=1.
\]
Now using \eqref{media1} and the fact that $e^{-\CS(\a)t} M_{\nu_t}(\a)$ is non-negative, it follows that the convergence of $e^{-\CS(\a)t} M_{\nu_t}(\a)$ holds in $L^1$ too.  
In view of Propositions 5.3 and 2.1 in \cite{BaLa} the law of $Z_\infty(\a)$ is a solution of  the fixed point equation \eqref{eq.fix.mix} and \eqref{media-delta} holds.

The proof of \eqref{betatozero} follows immediately from {(iii)} of Proposition \ref{Lemma2BL} and \eqref{convyule}.
\end{proof}
Denote by 
 $\CB$  the $\sigma$--field generated by the array of random variables 
$[\b_{jn}, j=1,\dots,n; n \geq 1]$. Given $\eps>0$ and $x_t \to +\infty$ as $t \to +\infty$, define the stochastic process
\[
 \Delta_t :=\sum_{n \geq 1} \J\{\nu_t=n\} P\Big \{ \big|\sum_{j=1}^n \beta_{jn} X_j\big |+\beta(n)|X_1| \leq \eps x_t e^{\qq(\a)t} 
 \Big | \CB \Big  \}.
\]
 
\begin{lemma}\label{lemma6} 
Let the same hypotheses of {\rm Theorem \ref{thm2}} or {\rm Theorem \ref{thm1}} be in force
for some $\a$ in $(0,2)$.
Then $\Delta_t \to 1$ in $L^1$ as $t \to +\infty$.
\end{lemma}

\begin{proof}
 Note that $0 \leq \Delta_t \leq 1$, hence
\[
 0 \leq \E[|\Delta_t-1|]=1-\E[\Delta_t].
\]
Furthermore
\[
 \E[\Delta_t]=P\{e^{-\qq(\a)t}(|V_t|+\beta(\nu_t)|X_1|)\leq \eps x_t\}.
\]
From Theorems \ref{thm2}-\ref{thm1} one knows that $e^{-\qq(\a)t} V_t$ converges in distribution. 
Moreover, from \eqref{betatozero}, one gets that $e^{-\qq(\a)t}\b(\nu_t)|X_1|$ converges in probability to zero.
Hence, $\big (e^{-\qq(\a)t}(|V_t|+\beta(\nu_t)|X_1|)\big )_{t \geq 0}$ is a tight family. This means that, for every sequence
$t_n \to +\infty$ and for every $\eta>0$, there exists $K$ such that 
$\inf_{n} P\{ e^{-\qq(\a)t_n}(|V_{t_n}|+\beta(\nu_{t_n})|X_1|)\leq K\} \geq 1-\eta$. Since $x_{t_n} \to +\infty$, 
for sufficiently large $n$ one can write
\[
 1\geq \E[\Delta_{t_n}] \geq P\{ e^{-\qq(\a)t_n}(|V_{t_n}|+\beta(\nu_{t_n})|X_1|)\leq  K\} \geq  1-\eta.
\]
Hence $\E[\Delta_t] \to 1$ and $\Delta_t \to 1$ in $L^1$.

\end{proof}

\begin{lemma}\label{Lemma5BL} 
If
$\CS(\a)<+\infty$,
one has
\begin{equation}\label{m2}
 \E[\tilde M_n(\a)^2]\leq C \sum_{i=1}^{n} i^{2\a(\qq(2\a)-\qq(\a))-1}
\end{equation}
for every $n$, $C$ being a suitable constant.
\end{lemma}

\begin{proof}
From the definition of $m_n(\a)$ we have $m_{n+1}(\a)=m_n(\a)(1+\frac{\CS(\a)}{n})$ and
from the definition of $\tilde M_n(\a)$ we obtain
\begin{displaymath}
\tilde M_{n+1}(\a)- \tilde M_{n}(\a)=-\frac{M_{n}(\a)}{m_n(\a)}\Big (\frac{\CS(\a)}{n+\CS(\a)} \Big ) +\sum_{i=1}^{n}\J\{I_n=j\} \frac{\beta_{jn}^\a(L_n+R_n-1)}{m_{n+1}(\a)}.
\end{displaymath}
Below the symbol $C$ designates a constant, not necessarily the same at each occurrence. 
\[
\begin{split}
|\tilde M_{n+1}(\a)- \tilde M_{n}(\a)|^2 & \leq 2\left(  \tilde M_{n}(\a)^2\Big (\frac{\CS(\a)}{n+\CS(\a)} \Big )^2+  \sum_{i=1}^{n}\J\{I_n=j\} \frac{\beta_{jn}^{2\a}(L_n+R_n-1)^2}{m_{n+1}(\a)^2}\right)\\
&\leq C \left[\frac{1}{n^2}\Big(  \sum_{i=1}^{n}  \frac{\beta_{jn}^\a}{m_{n}(\a)}\Big)^2  + \sum_{i=1}^{n}\J\{I_n=j\} \frac{\beta_{jn}^{2\a}(L_n+R_n-1)^2}{m_{n+1}(\a)^2}  \right]\\
&\leq C \left[\frac{1}{n} \sum_{i=1}^{n}  \frac{\beta_{jn}^{2\a}}{m_{n}(\a)^2}  + \sum_{i=1}^{n}\J\{I_n=j\} \frac{\beta_{jn}^{2\a}(L_n+R_n-1)^2}{m_{n+1}(\a)^2}  \right].
\end{split}
\]
Taking the expectation on both side of the last inequality we get
\[
\begin{split}
	\E(|\tilde M_{n+1}(\a)- \tilde M_{n}(\a)|^2) & \leq  \frac{C}{n} \left[ \frac{m_n(2\a)}{m_n(\a)^2}+ \frac{m_n(2\a)}{m_{n+1}(\a)^2}\right]\\
	&\leq \frac{C}{n} \Big[ n^{\CS(2\a)-2 \CS(\a)}+ \frac{n^{\CS(2\a)}}{(n+1)^{2\CS(\a)}}\Big].
\end{split}
\]
Now, recalling that $(\tilde M_{n}(\a))_{n\geq 1}$ is a martingale, we obtain
\[
\begin{split}
	 \E[\tilde M_n(\a)^2]&= 1+ \sum_{i=1}^{n-1}\E(|\tilde M_{i+1}(\a)- \tilde M_{i}(\a)|^2) \\
	& \leq C \sum_{i=1}^{n} i^{2\a(\qq(2\a)-\qq(\a))-1}.
\end{split}
\]
\end{proof}
\begin{lemma}\label{Lemma6new} 
If $\E[ L^{2\a}+R^{2\a}] <+\infty$, one has for every $t \geq 1$
\[
 e^{-2\CS(\a)t} \E[M_{\nu_t}(\a)^2] \leq \tilde h(t)
\]
where
\begin{equation}\label{defh}
\tilde h(t):=
\left\{%
\begin{array}{ll}
C_\mu & \text{if $\qq(2\a)<\qq(\a)$ and $2\CS(\a)>-1$;} \\
C_\mu h(t) & \text{otherwise} \\
\end{array}
\right.
\end{equation}
where $h(t)$ is defined in \eqref{defhA} and $C_\mu$ is a suitable constant.
\end{lemma}

\begin{proof}
As above the symbol $C$ designates a constant, not necessarily the same at each occurrence. 
We shall repeatedly use the following two simple facts:
for any $\gamma>-1$ and any $t>0$
\begin{equation}\label{formuletta} 
\sum_{n \geq 1} (1-e^{-t})^{n-1} n^\gamma \leq C e^{(\gamma+1)t}
\end{equation}
and, for every $t \geq 1$,
\begin{equation}\label{logaritmo} 
\sum_{n \geq 1} (1-e^{-t})^{n-1} \frac{1}{n}=\frac{t}{1-e^{-t}} \leq \frac{t}{1-e^{-1}}.
\end{equation}
Relation \eqref{logaritmo} follows by a simple Taylor expansion of $\log(1-x)$, while \eqref{formuletta} 
follows from \eqref{serie-gamma} and from the inequality
\[
 n^\gamma \leq C \frac{\Gamma (\gamma+n)}{\Gamma (n)\Gamma (\gamma+1)}.
\]
Since 
\[
 I_t:= \E[M_{\nu_t}(\a)^2]=
e^{-t} \sum_{n \geq 1} (1-e^{-t})^{n-1} m_n(\a)^2 \E[\tilde M_n(\a)^2],
\]
\eqref{asintm} and \eqref{m2} yield
\begin{equation}\label{statingpoint}
   I_t \leq C e^{-t} \sum_{n \geq 1} (1-e^{-t})^{n-1} n^{2\CS(\a)} \sum_{i=1}^n i^{2\a(\qq(2\a)-\qq(\a))-1}.
\end{equation}
Let $t \geq 1$. 
We need now to distinguish among different cases. 

\underline{Case 1.} If $\qq(2\a)<\qq(\a)$ and $2\CS(\a)>-1$, then $\sum_{i=1}^{+\infty} i^{2\a(\qq(2\a)-\qq(\a))-1}<+\infty$ and, 
by \eqref{formuletta}, one gets
\[
 I_t \leq C e^{-t} \sum_{n \geq 1} (1-e^{-t})^{n-1} n^{2\CS(\a)} \leq C e^{-t + t(2\CS(\a)+1)}
=Ce^{2\CS(\a)t}. 
\]

\underline{Case 2.} If $\qq(2\a)<\qq(\a)$ and $2\CS(\a)=-1$, then $\sum_{i=1}^{+\infty} i^{2\a(\qq(2\a)-\qq(\a))-1}<+\infty$ and,
by \eqref{logaritmo}, one gets
\[
 I_t \leq C e^{-t} \sum_{n \geq 1} (1-e^{-t})^{n-1} \frac{1}{n} 
\leq Ct. 
\]

\underline{Case 3.} If $\qq(2\a)<\qq(\a)$ and $2\CS(\a)<-1$, then $\sum_{i=1}^{+\infty} i^{2\a(\qq(2\a)-\qq(\a))-1}<+\infty$ and hence
\[
 I_t \leq C e^{-t} \sum_{n \geq 1} (1-e^{-t})^{n-1} n^{2\CS(\a)} \leq C e^{-t} \sum_{n \geq 1} n^{2\CS(\a)}
\leq C e^{-t}.
\]

\underline{Case 4.} If $\qq(2\a)>\qq(\a)$, noticing that $\sum_{i=1}^n i^{2\a(\qq(2\a)-\qq(\a))-1} \leq C n^{2\a(\qq(2\a)-\qq(\a))}$,
one gets
\[
  I_t \leq C e^{-t} \sum_{n \geq 1} (1-e^{-t})^{n-1} n^{2\a(\qq(2\a)-\qq(\a))},
\]
and then, by \eqref{formuletta},
\[
 I_t \leq Ce^{(\CS(2\a)-2\CS(\a))t}.
\]

\underline{Case 5.} If $\qq(2\a)=\qq(\a)$ and $0<\CS(\a)$ 
\[
\begin{split}
 I_t &\leq C  e^{-t}  \sum_{n \geq 1} (1-e^{-t})^{n-1} n^{2\CS(\a)} \log n  \\
& \leq  C  e^{-t}  \sum_{n \geq 1} (1-e^{-t})^{n-1} n^{2\CS(\a)+\eta} 
=C_\eta e^{\eta t}. \\
\end{split}
\]

 If $\qq(2\a)=\qq(\a)$ and  $2\CS(\a) \leq 0 $, then 
\[
\begin{split}
 I_t & \leq C  e^{-t} \sum_{n \geq 1} (1-e^{-t})^{n-1} n^{2\CS(\a)} \sum_{i=1}^n i^{-1} \\
     &  =  C  e^{-t}   \sum_{i \geq 1} i^{-1}  \sum_{n \geq i} (1-e^{-t})^{n-1} n^{2\CS(\a)} \\
      &  \leq   C  e^{- t}   \sum_{i \geq 1} i^{-1} (1-e^{-t})^{i-1}   \sum_{k \geq 0} (1-e^{-t})^{k} (k+1)^{2\CS(\a)}\\
\end{split}
\]
Hence:

\underline{Case 6.} If $\qq(2\a)=\qq(\a)$ and $2\CS(\a) <-1$, by \eqref{logaritmo}
\[
 I_t \leq  C  e^{- t}   \sum_{i \geq 1} i^{-1} (1-e^{-t})^{i-1}\sum_{k \geq 1} k^{2\CS(\a)} =  C  e^{- t}  
  \sum_{i \geq 1} i^{-1} (1-e^{-t})^{i-1}
 \leq C t e^{-t}  
\]

\underline{Case 7.} If $\qq(2\a)=\qq(\a)$ and $2\CS(\a) =-1$, using \eqref{logaritmo} twice 
\[
  I_t \leq C  t^2 e^{-t}. 
\]

\underline{Case 8.} If $\qq(2\a)=\qq(\a)$ and $-1<2\CS(\a) \leq 0$, by \eqref{formuletta} and \eqref{logaritmo},
\[
 I_t \leq C  e^{-t}  \sum_{i \geq 1} i^{-1}(1-e^{-t})^{i-1} \sum_{k \geq 0} (1-e^{-t})^{k} (k+1)^{2\CS(\a)}
=Ct e^{2\CS(\a)t}
\]

\end{proof}

\subsection{Proofs of the main theorems}

\begin{proof}[Proof of Theorem \ref{thm1}]
The proof follows the same steps of the one of Theorem 2.2 in \cite{BaLa}, using in place 
of Lemma 5.1  in \cite{BaLa} the following simple result:  
{\it Let $(X_{n})_{n \geq 1}$ be a sequence of iid random variables with common distribution function $F_0$. 
Assume that $(a_{jn})_{j \geq 1, n \geq 1}$ is an array of positive weights such that 
\[
 \lim_{n \to +\infty} \sum_{j=1}^n a_{jn} =a_\infty \qquad \text{and} \qquad \lim_{n \to +\infty} \max_{1 \leq j \leq n} a_{jn}=0.
\]
If $F_0$ satisfy \eqref{stabledomain} with $\a=1$, $c_0^+=c_0^->0$ and 
\eqref{gamma} holds, then $\sum_{j=1}^n a_{jn} X_j$ converges in law to a Cauchy random variable of scale parameter $\pi a_\infty c_0$ and position parameter $a_\infty \gamma_0$.}
To prove this claim, according to the classical general central limit theorem for array of independent random variables, it is enough to prove that  
  \begin{align}
    \label{condition1-uno-bis}
    &\lim_{n \to +\infty} \z_{n}(x) 
    = \frac{a_\infty c_0}{|x|} \qquad (x\not =0), \\
    \label{clt3condtris}
    &\lim_{\eps \to 0^+}\lim_{n \to +\infty} \s^{2}_{n}(\eps) = 0 , \\
    \label{condition3-bis}
    &\lim_{n \to +\infty} \eta_{n} = a_\infty \gamma_0 
  \end{align}
  are simultaneously satisfied where
  \begin{align*}
    \z_n(x) &:= 
    \J\{x<0\} \sum_{j=1}^{n} Q_{j,n}(x) + \J\{x>0\} \sum_{j=1}^{n} (1-Q_{j,n}(x)) \qquad (x \in \RE), \\
    \s_{n}^2(\eps) &:=  
    \sum_{j=1}^{n}\Big\{ \int_{(-\eps,+\eps]} x^2\,dQ_{j,n}(x)-\Big( \int_{(-\eps,+\eps]} x\,dQ_{j,n}(x)\Big)^2 \Big\} \qquad (\eps>0), \\
    \eta_{n} &: =   \sum_{j=1}^{n}\Big\{ 1- Q_{j,n}(1) -Q_{j,n}(-1) + \int_{(-1,1]} x\, dQ_{j,n}(x)\Big \}, \\
    Q_{j,n}(x)&:=F_0\big(a_{j,n}^{-1}x\big) \quad \text{with the convention  $F_0(\cdot/0):=\J_{[0,+\infty)}(\cdot)$}. \\
  \end{align*}
See, e.g., Theorem 30 and Proposition 11 in \cite{fristedgray}. Conditions \eqref{condition1-uno-bis} and \eqref{clt3condtris}
can be proved exactly as the analogous conditions of Lemma 5 in \cite{BassLadMatth10}. As for condition \eqref{condition3-bis} note that
\[
\eta_{n}=\sum_{j=1}^n a_{jn} \int_{(-1/a_{jn},1/a_{jn}]} x dF_0(x)+
\sum_{j=1}^n a_{jn}\left[\Big(1-F_0\Big(\frac{1}{a_{jn}}\Big)\Big)\frac{1}{a_{jn}}-F\Big(-\frac{1}{a_{jn}}\Big)\frac{1}{a_{jn}}\right].
\]
Using the assumptions on $F_0$ and on $(a_{jn})_{jn}$ it follows immediately that
\[
\lim_n \sum_{j=1}^n a_{jn} \int_{(-1/a_{jn},1/a_{jn}]} x dF_0(x)= a_\infty \gamma_0
\]
and
\[
 \qquad \lim_n \sum_{j=1}^n a_{jn}\left[\Big(1-F_0\Big(\frac{1}{a_{jn}}\Big)\Big) \frac{1}{a_{jn}}-F\Big(-\frac{1}{a_{jn}}\Big)\frac{1}{a_{jn}}\right]=a_\infty(c^+_0-c^-_0)=0.
\]
This gives  \eqref{condition3-bis}.

\end{proof}

\begin{proof}[Sketch of the proof of Theorem \ref{thm.eqdiff}]
Using the results in \cite{Kielek} one proves that
\[
q_t:=\sum_{n \geq 1} e^{-t}(1-e^{-t})^{n-1} \tilde q_n,
\]
where $\tilde q_1:=\bar \rho_0$ and
\[
\tilde q_n:=\frac{1}{n} \sum_{i=0}^{n-1} \tilde Q^+(\tilde q_i,\tilde q_{n-1-i}) \qquad (n \geq 1),
\]
is a solution of an homogeneous kinetic equation of the form \eqref{eq.1} 
with $Q^+$  replaced by $\tilde Q^+$. At this stage, following
the same arguments used to prove Proposition 1 in \cite{BassLadMatth10}, one proves that $q_t$ is the law of $H_t$. 
\end{proof}

\begin{proof}[Proof of Theorem \ref{maxlimittheorem}]
Let $x>0$ and let $\CB^*$ the $\sigma$-field generated by the array 
of weights $[\beta_{jn}]_{j,n}$ and by the Yule process $[\nu_t]_{t \geq 0}$. Then
\begin{equation}\label{stimabasemax}
 \begin{split}
P\{e^{-\qq(\a)t} \frH_t \leq x \}&=\E\left [ \prod_{j=1}^{\nu_t} P\{  | e^{-\qq(\a)t} \beta_{j,\nu_t} X_j| \leq x|\CB^* \}\right ] \\
&=\E\left[ \prod_{j=1}^{\nu_t} \left(1-P\{  | e^{-\qq(\a)t} \beta_{j,\nu_t} X_j| > x|\CB^* \}\right)\right] \\
&=\E\left[e^{-\frac{c_0}{x^\a} e^{-\CS(\a)t}M_{\nu_t}(\a)}+\Lambda_t(x)\right] \\
 \end{split}
\end{equation}
where
\[
 \Lambda_t(x) := \prod_{j=1}^{\nu_t} (1-P\{  | e^{-\qq(\a)t} \beta_{j,\nu_t} X_j| > x|\CB^* \})-\prod_{j=1}^{\nu_t} e^{-\frac{c_0}{x^\a} e^{-\CS(\a)t} \beta_{j\nu_t}^\a }. 
\]
By \eqref{eq.8}
\begin{equation}\label{eq.8max}
 P\{  | e^{-\qq(\a)t} \beta_{j,\nu_t} X_j| > x|\CB^* \})=\frac{c_0e^{-\CS(\a)t} \beta_{j,\nu_t}^\a}{  x^\a}\Big (1+ \CR\Big(\frac{x}{e^{-\qq(\a)t} \beta_{j,\nu_t}} \Big) \Big ).
\end{equation}
Now  recall that, given $2N$
complex numbers $a_1,\dots,a_N$,$b_1$,$\dots,b_N$
with $|a_i|, |b_i| \leq 1$, 
\(
|\prod_{i=1}^N a_i - \prod_{i=1}^N b_i| \leq \sum_{i=1}^N |a_i-b_i|.
\)
Moreover, for every $x>0$ and $1 \leq r \leq 2$ one has
\(
 |1-x-e^{-x}| \leq C_r |x|^r.
\)
Combining these facts with \eqref{eq.8max} one gets
\[
\begin{split}
 |\Lambda_t(x)| & \leq \sum_{j=1}^{\nu_t} \left  | 1-P\{  | e^{-\qq(\a)t} \beta_{j,\nu_t} X_j| > x|\CB^* \} -e^{-\frac{c_0}{x^\a} e^{-\CS(\a)t} \beta_{j\nu_t}^\a } \right | \\
   &=  \sum_{j=1}^{\nu_t} \left |1-\frac{c_0e^{-\CS(\a)t} \beta_{j,\nu_t}^\a}{  x^\a} \Big (1+ \CR\Big(\frac{x}{e^{-\qq(\a)t} \beta_{j,\nu_t}} \Big) \Big ) -
e^{-\frac{c_0}{x^\a} e^{-\CS(\a)t} \beta_{j\nu_t}^\a } \right |    
\\
& 
\leq C_r c_0^r  e^{-r\CS(\a)t} \sum_{j=1}^{\nu_t} \frac{\beta_{j,\nu_t}^{r\a }}{x^{r\a }}  + \sum_{j=1}^{\nu_t} \frac{c_0e^{-\CS(\a)t} \beta_{j,\nu_t}^\a}{  x^\a} 
  \Big|\CR\Big(\frac{x}{e^{-\qq(\a)t} \beta_{j,\nu_t}} \Big)\Big|  \\
& \leq
  \frac{C_rc_0^r}{x^{r\a}} e^{-r\CS(\a)t} M_{\nu_t}(r\a)+\frac{c_0}{x^\a} e^{-\CS(\a)t} M_{\nu_t}(\a) \bar \CR\Big(\frac{x}{e^{-\qq(\a)t} \beta_{(\nu_t)}} \Big)
\\
&=\frac{C_rc_0^r}{x^{r\a}} e^{-r\a (\qq(\a)-\qq(r\a))t}  e^{-\CS(r\a)t} M_{\nu_t}(r\a)+\frac{c_0}{x^\a} e^{-\CS(\a)t} M_{\nu_t}(\a) \bar \CR\Big(\frac{x}{e^{-\qq(\a)t} \beta_{(\nu_t)}} \Big)
\\
\end{split}
\]
Now choose $r=\delta/\a$ and notice that $r>1$. Moreover, by the convexity of $\CS(s)$, it is easy to see that $\mu(s)<\mu(\a)$ if $\a<s<\delta$. Hence, without loss of generality, we can suppose that $\a<\d<2$.
Then, arguing as in the proof of \eqref{media1} of Proposition \ref{Lemma7new}, it is immediate to see that it also holds
\[
 \E[e^{-\CS(r\a)t} M_{\nu_t}(r\a)]=1. 
\]
Moreover, by assumption, $\qq(\a)-\qq(\delta)=\qq(\a)-\qq(r\a)>0$, hence
\begin{equation*}\label{convzeroR1a}
 \E[e^{-r\a (\qq(\a)-\qq(r\a))t}  e^{-\CS(r\a)t} M_{\nu_t}(r\a)] \to 0
\end{equation*}
when $t \to +\infty$. Combining \eqref{convL1} and \eqref{betatozero} by the generalized dominated convergence theorem one gets also that
\begin{equation*}\label{convzeroR1b}
  \E\left [e^{-\CS(\a)t} M_{\nu_t}(\a) \bar \CR\Big(\frac{x}{e^{-\qq(\a)t} \beta_{(\nu_t})} \Big)\right] \to 0.
\end{equation*}
Hence
\(
\E[\Lambda_t(x)] \to 0 
\)
as $t \to +\infty$. 
Using once again \eqref{convL1} one gets
\[
 \E\left[e^{-\frac{c_0}{x^\a} e^{-\CS(\a)t} M_{\nu_t}(\a)} \right ]\to \E\left[e^{-\frac{c_0}{x^\a} Z_{\infty}(\a)}\right ].
\]
Plugging these last convergences in \eqref{stimabasemax}
one concludes the proof for $x>0$. Since for $x<0$ there is nothing to prove, let us assume that $x=0$.
By dominated convergence theorem it is easy to see
that
\[
 \lim_{x \downarrow 0}  \E[ e^{-|x|^{-\a} Z_{\infty}{(\a)}} ] =P\{Z_{\infty}{(\a)}=0\}.
\]
Hence, if $P\{Z_\infty(\a)=0\}>0$ there is nothing to be proved since 
$0$ is a discontinuity point for $x \mapsto \E[ e^{-|x|^{-\a} Z_{\infty}{(\a)}} ]=:\fHH_\infty(x)$.
Assuming now $P\{Z_\infty(\a)=0\}=0$, one obtains that $\fHH_\infty$ is continuous and that 
for every $\eps>0$ there is $\eta=\eta(\eps)$ such that $\fHH_\infty(\eta)\leq \eps$. So that
\[
0 \leq \limsup_{t \to +\infty} \fHH_t(0)\leq \limsup_{t \to +\infty} \fHH_t(\eta)= \fHH_\infty(\eta) \leq \eps.
\]
This proves that $ \lim_{t \to +\infty} \fHH_t(0)=0$.

\end{proof}

\begin{proof}[Proof of Theorem. \ref{main.thm}]
Recalling that $\CB$ denotes the $\sigma$--field generated by the array of random variables 
$[\b_{jn}]_{jn}$, using \eqref{lowerbound} one can write
\[
 x_t^\alpha P\{|e^{-\mu(\a)t}V_t|>x_t\}=x_t^\alpha \E\Big[\sum_{n \geq 1} \J\{\nu_t=n\} P\Big\{|\sum_{j=1}^{n} e^{-\qq(\a)t}\beta_{jn}X_j |>x_t|\CB\Big\}\Big]
\geq B_t^{(0)}-B_t^{(1)}
\]
where
\[
\begin{split}
 B_t^{(0)} &:= e^{-\CS(\a)t}\E\Big[\Delta_t  \Big[\Big(1-\bar \CR\Big( \frac{x_t(1+\eps)}{\beta(\nu_t)e^{-\qq(\a)t}}\Big)\Big) \vee 0\Big ]\sum_{j=1}^{\nu_t} \beta_{j\nu_t}^\a \Big] \frac{c_0}{(1+\eps)^\a}\\
 B_t^{(1)} &:= 
e^{-2\CS(\a)t}\frac{K_0^2}{x_t^\a(1+\eps)^{2\alpha}}\E\Big[( M_{\nu_t}(\a)^2\Big ],
\end{split}
\]
for every $\eps>0$.
Setting
\[
D_t:= \Big[\Big(1-\bar \CR\Big( \frac{x_t(1+\eps)}{\beta(\nu_t)e^{-\qq(\a)t}}\Big)\Big) \vee 0\Big ]M_{\nu_t}(\a)e^{-\CS(\a)t}
\]
one gets that for every $t>0$
\[
|D_t|\leq (1+||\CR||_{\infty})M_{\nu_t}(\a)e^{-\CS(\a)t}
\]
and by \eqref{convL1}  $M_{\nu_t}(\a)e^{-\CS(\a)t}\to Z_\infty(\a)$ in $L^1$. Furthermore $|\Delta_t|\leq 1$ and $\Delta_t \to 1$ in probability by Lemma \ref{lemma6}. Finally by \eqref{betatozero} and by \eqref{eq.8bis}, one gets
\[
\bar \CR\left( \frac{x_t(1+\eps)}{\beta(\nu_t)e^{-\qq(\a)t}}\right)\to 0
\]   
in probability. Combining these facts one obtains that
\[
D_t \to Z_\infty(\a) \qquad\text{and}\qquad \Delta_tD_t \to Z_\infty(\a)
\]
in probability for $t\to +\infty$ and, by the generalized dominated convergence theorem, that $\Delta_tD_t \to Z_\infty(\a)$ in $L^1$. Hence, in view of \eqref{media1} one obtains
\[
\begin{split}
\lim_{t\to+\infty}B_t^{(0)}&=\lim_{t\to+\infty}\frac{c_0}{1+\eps^\a}\E(\Delta_t D_t) \\
&=\frac{c_0}{1+\eps^\a}\E\big(Z_\infty(\a)\big)\\
&=\frac{c_0}{1+\eps^\a}.
\end{split}
\]

As far as the term $B_t^{(1)}$ is concerned, using Lemma \ref{Lemma6new}, one can write
\[
\limsup_{t \to +\infty} B_t^{(1)} \leq C \limsup_{t \to +\infty} \frac{\tilde h(t)}{x_t^{\alpha}}
\]
for a suitable constant $C$ and $\tilde h(t)$ being defined as in the same lemma.
Then, in view of the assumptions on $x_t$ according to the expression of $\tilde h(t)$, it follows that $\limsup_{t \to +\infty} B_t^{(1)} =0$.
Hence, one gets
\[
\liminf_{t \to +\infty} x_t^\a P\{|V_t| >x_t\} \geq \liminf_{t \to +\infty} B_t^{(0)} - 
\limsup_{t \to +\infty} B_t^{(1)}= \frac{c_0}{(1+\eps)^{\a}}
\]
and then 
\begin{equation}\label{eq.23}
\liminf_{t \to +\infty} x_t^\a P\{|V_t| >x_t\} \geq c_0.
\end{equation}

On the other hand, applying \eqref{upperbound}, one gets
\begin{equation}\label{upperbound2}
\begin{split}
x_t^\alpha P\{|e^{-\nu(\a)t}V_t|>x_t\} & \leq \frac{c_0 }{(1-\eps)^\alpha} \E\left [\left (1+\bar \CR \left(\frac{x_t(1-\eps)}{\beta(\nu_t)e^{-\qq(\a)t}}\right)\right) e^{-\CS(\a)t}M_{\nu_t}(\a)  \right] \\
&
+\frac{2K_0}{\eps^2(2-\alpha) x_t^{(2-\alpha)(1-\gamma)}}  \E \left[ e^{-\CS(\a)t}M_{\nu_t}(\a)) \right]
\\ 
&  
+   \left [ \frac{ K_0 ^2}{x_t^{\alpha(2\gamma-1)}} +
\frac{K_1}{\eps^2 x_t^{2-\alpha+2(\alpha-1)\gamma}} \right]  \E\left [\big(e^{-\CS(\a)t}M_{\nu_t}(\a)\big)^2\right ]
\\
&=:U_t^{(0)}+U^{(1)}_t+U^{(2)}_t.
\end{split}
\end{equation}
As in the previous part $U_t^{(0)} \to c_0/(1-\eps)^\a$. Moreover, since $(2-\alpha)(1-\gamma)>0$ for every $\gamma<1$  and
 $\E[ e^{-\CS(\a)t}M_{\nu_t}(\alpha) ]=1$ by \eqref{media1}, one has $U^{(1)}_t \to 0$ for $t \to +\infty$. 
Finally, in view of Lemma \ref{Lemma6new}, Remark~\ref{remark1} and the assumptions on $x_t$, 
according to the value of $\CS(\a)$ and $\CS(2\a)$, one can choose $1/2<\gamma<1$ in order that $U^{(2)}_t \to 0$ for $t \to +\infty$.
Hence, 
$\limsup_{t \to +\infty} x_t^\a P\{|e^{-\CS(\a)t}V_t| > x_t\} \leq c_0/(1-\eps)^\a$ for every $\eps>0$ which implies 
\begin{equation}\label{eq.24}
 \limsup_{t \to +\infty} x_t^\a P\{|e^{-\CS(\a)t}V_t| > x_t\} \leq c_0.
\end{equation}
In view of \eqref{eq.23} and \eqref{eq.24}  we obtain
\begin{equation}\label{limite}
\lim_{t \to +\infty} x_t^\a P\{|e^{-\CS(\a)t}V_t| > x_t\} = c_0.
\end{equation}
In  order to complete the proof of \eqref{eq.7} it is sufficient to show that
\begin{equation}\label{eq.v-inf}
\lim_{t \to +\infty} x_t^\a P\{|V_{\infty}| > x_t\} = c_0.
\end{equation}
As already noted, by convexity of $\CS$  and the condition $\qq(\delta)<\qq(\a)$, 
it follows that $\qq(s)<\qq(\a)$ if $\a<s<\d$.
Hence, without loss of generality, we can assume that $\a<\d<2\a$.

Let $Z_{\infty}(\a)$ be as in Theorems \ref{thm2} and \ref{thm1}. Then
\[
V_{\infty}\stackrel{\CL}{=}Z_{\infty}(\a)^{1/\a}S_{\a}
\]
where $S_{\a}$ is a stable r.v. with index $\a$, $Z_{\infty}(\a)$ and $S_{\a}$
being independent. If $F_{\infty}$ and $G_\a$ denote the distribution functions of $V_{\infty}$ and $S_{\a}$, respectively, then
\[
F_{\infty}(x)=\E\Big[G_\a(Z_{\infty}(\a)^{-1/\a}x)\J\{Z_{\infty}(\a)\neq 0\}.\Big]
\]
 Hence
\begin{equation}\label{code}
P\{|V_{\infty}| > x\} = F_\infty(-x)+1-F_\infty(x)
=:\frac{c_0}{x^\a} \E(Z_{\infty}(\a))+ \zeta(x)=\frac{c_0}{x^\a}+ \zeta(x)
\end{equation}
since $\E(Z_{\infty}(\a))=1$ by \eqref{media1}.
From the properties of the tails of stable distributions one can write that 
\[
\Big|G_\a(-x)+1- G_\a(x) -\frac{c_0}{x^\a}\Big|\leq \frac{K}{x^\d}
\]
for $x>0$, since $\a<\d<2\a$. See, e.g., \cite{Ibragimov}.
Hence 
\[
\zeta(x)\leq C\frac{\E(Z_{\infty}(\a))^{\d/\a}}{x^\d}
\]
with $\E[Z_{\infty}(\a)^{\d/\a}]<+\infty$ by \eqref{media-delta}.

To prove \eqref{eq.7tre},  use \eqref{maxineq} to write
\[
\tilde B_t^{(0)}- \tilde B_t^{(1)} \leq  x_t^\alpha P\{|e^{-\mu(\a)t}\frH_t|>x_t\} \leq \tilde U_t^{(0)}
\]
where
\[
\begin{split}
 \tilde B_t^{(0)} :&
= c_0 \E\Big[\Big(1-\bar \CR\Big( \frac{x_t}{\beta(\nu_t)e^{-\qq(\a)t}}\Big)\Big)e^{-\CS(\a)t} M_{\nu_t}(\a) \Big] \\
\tilde  B_t^{(1)} :&= 
\frac{K_0^2}{x_t^\a}\E\Big[\big(e^{-\CS(\a)t}M_{\nu_t}(\a)\big)^2\Big ]\\
 \tilde U_t^{(0)}:&= c_0 \E\left [\left (1+\bar \CR \left(\frac{x_t}{\beta(\nu_t)e^{-\qq(\a)t}}\right)\right)  e^{-\CS(\a)t} M_{\nu_t}(\a)\right] 
\end{split}
\]
 Arguing as before, one proves that $\tilde B_t^{(0)} \to  c_0$, $\tilde B_t^{(1)} \to 0$ and $\tilde U_t^{(0)} \to c_0$ and this completes the proof. 
\end{proof}

\end{document}